\documentclass[twoside]{deon16}

\usepackage{deon16macro}

\usepackage{graphicx}
\usepackage{amsmath}
\usepackage{amssymb}
\usepackage[utf8]{inputenc}
\newcommand{\si}{\Rightarrow}
\newcommand{\sip}{\Rightarrow_{\prec}}
\newcommand{\precs}{\prec_{\Rightarrow}}
\usepackage{bussproofs}
\EnableBpAbbreviations
\newcommand{\sci}{>\!\!=}
\newcommand{\scip}{{\sci}_\prec}

\newcommand{\oastp}{\oast_{\npcon}}
\newcommand{\pcono}{\pcon_{\oast}}
\newcommand{\odotd}{\odot_{\nppcon}}
\newcommand{\dodot}{{\ppcon}_\odot}
\newcommand{\sn}{\neg}
\newcommand{\scn}{{\sim}}
\usepackage{graphicx}
\usepackage{amsmath}
\usepackage{amssymb}
\usepackage{amsfonts}
\usepackage{url}
\usepackage[all]{xy}
\usepackage{makeidx}
\usepackage{latexsym}
\usepackage{empheq} 
\usepackage{verbatim}
\usepackage{stmaryrd}
\usepackage{enumitem}
\usepackage{xspace}
\usepackage{bussproofs}
\usepackage{boxedminipage}
\usepackage{hyperref}
\usepackage{ccaption}
\usepackage{cancel}
\usepackage{multicol}
\usepackage{xcolor}
\usepackage{proof}
\usepackage{longtable}
\usepackage{marginnote}

\newcommand{\pcon}{\mathcal{C}}
\newcommand{\npcon}{\cancel{\mathcal{C}}}
\newcommand{\ppcon}{\mathcal{D}}
\newcommand{\nppcon}{\cancel{\mathcal{D}}}

\newcommand{\wt}{{\rhd}}
\newcommand{\bt}{{\blacktriangleright}}
\newcommand{\tw}{{\lhd}}
\newcommand{\tb}{{\blacktriangleleft}}
\usepackage{forest}
\usepackage{cmll}
\usetikzlibrary{patterns}

\def\fns{\footnotesize}

\newcommand{\nc}{\,\mid\!\sim}

\def\aol{\rule[0.5865ex]{1.38ex}{0.1ex}}

\def\pdra{\mbox{$\,>\mkern-8mu\raisebox{-0.065ex}{\aol}\,$}}

\tikzset{
	treenode/.style = {align=center, inner sep=0pt, text centered},
	Ske/.style = {treenode, ellipse, double, draw=black,
		minimum width=6pt, thick},
	PIA/.style = {treenode, ellipse, black, draw=black,
		minimum width=6pt},
	Crit/.style = {treenode, rectangle, draw=black,
		minimum width=0.5em, minimum height=0.5em}
}

\def\lastname{De Domenico, Farjami, Manoorkar, Palmigiano, Panettiere, Tzimoulis, Wang}

\begin{document}

\begin{frontmatter}
  \title{
Normative implications}
  \author{Andrea De Domenico}
  \address{Vrije Universiteit Amsterdam}
  \author{Ali Farjami}
  \address{University of Luxembourg}
  \author{Krishna Manoorkar}
  \address{Vrije Universiteit Amsterdam}
  \author{Alessandra Palmigiano}
  \address{Vrije Universiteit Amsterdam, Department of Mathematics and Applied Mathematics, University of Johannesburg}
  \author{Mattia Panettiere}
  \address{Vrije Universiteit Amsterdam}
  \author{Apostolos Tzimoulis}
  \address{University of Luxembourg}
  \author{Xiaolong Wang}
  \address{Vrije Universiteit Amsterdam}

  \begin{abstract}
  We  continue to develop a research line initiated in \cite{wollic22}, studying I/O logic from an algebraic approach based on subordination algebras. 
  We introduce the classes of slanted (co-)Heyting algebras, as equivalent presentations of distributive lattices with subordination relations. Interpreting subordination relations as the algebraic counterparts of input/output relations on formulas yields (slanted) modal operations with interesting deontic interpretations. We study the theory of slanted and co-slanted Heyting algebras, develop algorithmic correspondence and inverse correspondence, and present some deontically meaningful axiomatic extensions and examples.
  \end{abstract}
\begin{keyword}
  I/O logic,
modal characterizations of normative conditions,
 subordination algebras,
slanted Heyting algebras, 
(inverse) correspondence.
 \end{keyword}
 \end{frontmatter}

\section{Introduction}
\label{sec: introduction}
This paper continues a line of investigation, initiated in \cite{wollic22}, 
on the study of input/output logic \cite{Makinson00}  from an algebraic perspective based on   {\em subordination algebras} \cite{de2024obligations2}. This approach has allowed to uniformly extend input/output logic to a large family of nonclassical logics \cite{de2024obligations}, and, in this generalized context, to obtain modal characterizations of an infinite class of conditions on normative and permission systems as well as on their interaction \cite{de2024correspondence}. Subordination algebras are tuples $(A, \prec)$ where $A$ is an algebra (typically, a Boolean algebra), and ${\prec}\, \subseteq A\times A$ is a {\em subordination relation}, i.e.~a binary relation endowed with the algebraic counterparts of the well known properties $(\top), (\bot), (\mathrm{SI}), (\mathrm{WO}), (\mathrm{AND}), (\mathrm{OR})$ of normative systems in input/output logic.
These properties can be equivalently reformulated as the requirement that, for every $a\in A$, the sets ${\prec}[a] := \{b\in A\mid a\prec b\}$ and ${\prec}^{-1}[a] := \{b\in A\mid b\prec a\}$ be a filter and an ideal of $A$, respectively.
Subordination algebras have been  introduced independently from input/output logic,  in the context of a research program in point-free topology, aimed at developing region-based theories of space. In this literature, subordination algebras are used as an `umbrella' type of notion which crops up under several  equivalent presentations. One of these presentations is the notion of {\em quasi-modal algebras} introduced by Celani \cite{celani2001quasi} (see discussions in  \cite{de2021slanted}\cite{de2024obligations2}). These are tuples $(A, \Delta)$ such that $A$ is a Boolean algebra and $\Delta$ is a {\em quasi-modal operator}, i.e.~a map such that $\Delta(a)$ is an ideal of $A$ for any $a\in A$. Structures closely related to quasi modal algebras are {\em generalized implication lattices}   \cite{castro2011distributive}, i.e.~tuples $(A, \Rightarrow)$ such that $A$ is a distributive lattice and $\Rightarrow$ is a {\em generalized implication}, i.e.~a binary map such that $a\Rightarrow b$ is an ideal of $A$ for all $a, b\in A$, and satisfying certain additional conditions. Mediated by the notion of quasi-modal operator, in \cite[Lemma 5]{calomino2023study}, it is shown that generalized implications and subordination relations over a given Boolean algebra bijectively correspond to each other. 
 This observation is one of the starting points of the present paper.

In the present paper, we generalize the connection between generalized implications and subordination relations in the context of (distributive) lattices by introducing  {\em slanted Heyting algebras} and {\em slanted co-Heyting algebras} (see Definitions \ref{def:slanted Heyting} and \ref{def:slanted co-Heyting}). Any subordination algebra $\mathbb{S}=(A, \prec)$ induces the binary slanted operators $\sip$ and $\scip$  on $A$ such that each $a, b\in A$ are mapped to the following open and closed elements of its canonical extensions, respectively:
{{\small
\begin{equation}
\label{eq:slantedheyting coheyting}
a\sip b: = \bigvee\{c\in A\mid a\wedge c\prec b\}
\quad \mbox{and} \quad a\scip b: = \bigwedge\{c\in A\mid b\prec a\vee c\}.
\end{equation}
}}
These slanted operations  can be understood as {\em normative} counterparts of the identities  defining Heyting implication and co-implication, respectively:
{\small \begin{equation}
\label{eq:heyting coheyting}
    a \rightarrow b: = \bigvee\{c\in A\mid a\wedge c \leq b\}
\quad \mbox{and} \quad a \pdra  b: = \bigwedge\{c\in A\mid b\leq a\vee c\}
\end{equation}}
 \noindent indeed, the order relation in \eqref{eq:heyting coheyting}, encoding {\em logical} entailment, is replaced in \eqref{eq:slantedheyting coheyting} by the subordination  relation $\prec$ which  encodes {\em normative} entailment.  We can  interpret $a \sip b$  
 as `the disjunction of all propositions that  {\em normatively} imply $b$ when in conjunction with $a$'  (and $a \scip b$  as `the conjunction of all propositions whose disjunction with $a$ is {\em normatively} implied by $b$').
%
Hence, $a \sip b$ can be understood as the weakest side condition, or {\em context}, under which $a$ normatively implies $b$. For  example, consider the conditional  obligation $citizen\prec taxes$, which reads `If you are a citizen then you must pay taxes'. This   obligation holds under a set of background assumptions  such as  `You earn more than a minimum threshold' ($earn$),  which are  often left implicit. The expression $citizen \sip taxes$ allows us to represent, and make inferences with, the {\em constellation} of background conditions which make $citizen\prec taxes$ a valid conditional obligation, purely in terms of $citizen$ and $taxes$; for example, in this case, we can represent this scenario by the inequality (`logical entailment')  $earn\leq  citizen\sip taxes$.
 In this way, the language enriched with  $\sip$ becomes more suitable to express flexible dependencies that reflect real-world scenarios where obligations and permissions may change based on context, capacity, or other factors.

\paragraph{Structure of the paper.} In Section \ref{sec: preliminaries}, we collect basic technical definitions; in Section \ref{sec:slanted Heyting}, we introduce slanted Heyting algebras and show that they equivalently represent subordination algebras; in Section \ref{sec:examples on sip}, we discuss how axioms in the language of slanted Heyting algebras  capture interesting normative conditions; in Sections \ref{sec: slanted co-heyting} and \ref{sec:slanted negations}, we introduce slanted Heyting co-implication and the pseudo (co-)complements; more examples of conditions are discussed in Section \ref{sec:more examples}; in Section \ref{sec: correspondence inverse corr}, we identify the classes of axioms and normative conditions that correspond to each other; in Section \ref{sec:Models} we discuss how this language can be used to capture norms in various types of contexts; we conclude in Section \ref{sec:conclusions}.

\section{Preliminaries}
\label{sec: preliminaries}
  In what follows,  when we say `lattice', we mean `bounded lattice'.  
Let $A$ be a sublattice 
of a complete lattice $A'$. 
\begin{enumerate}[noitemsep,nolistsep]
\item An element $k\in A'$ is {\em closed} if $k = \bigwedge F$ for some non-empty 
$F\subseteq A$; an element $o\in A'$ is {\em open} if $o = \bigvee I$ for some non-empty 
$I\subseteq A$;  
\item  $A$ is {\em dense} in $A'$ if every element of $A$ is both  the join of closed elements and  the meet of open elements of $A'$.
\item $A$ is {\em compact} in $A'$ if, for all nonempty $F, I\subseteq A$, 
if $\bigwedge F\leq \bigvee I$ then $\bigwedge F'\leq \bigvee I'$ for some finite $F'\subseteq  F$ and some finite $I'\subseteq I$. 
\item The {\em canonical extension} of a lattice $A$ is a complete lattice $A^\delta$ containing $A$
as a dense and compact sublattice.
\end{enumerate}
The canonical extension $A^\delta$ of any lattice $A$ always exists and is  unique up to an isomorphism fixing $A$ (cf.~\cite[Propositions 2.6 and 2.7]{DuGePa05}).


%
We let $K(A^\delta)$ (resp.~$O(A^\delta)$) denote the set of the closed (resp.~open) elements    of $A^\delta$. It is easy to see that $A = K(A^\delta)\cap O(A^\delta)$, which is why the elements of $A$ are referred to as the {\em clopen} elements of $A^\delta$. The following propositions collect well known facts which we will use in the remainder of the paper. In particular, item  (iv) of the next proposition is a variant of \cite[Lemma 3.2]{gehrke2001bounded}.

 \begin{proposition} \label{prop:background can-ext} (cf.~\cite[Proposition 2.6]{de2024obligations2})
 For any lattice $A$,
  all $k_1, k_2\in K(A^\delta)$, $o_1, o_2\in O(A^\delta)$, and  $u_1, u_2\in A^\delta$,
 \begin{enumerate}[label=(\roman*), noitemsep,nolistsep]
     \item $k_1\leq k_2$ iff $k_2\leq b$ implies $k_1\leq b$ for all $b\in A$.
     \item $o_1\leq o_2$ iff $b\leq o_1$ implies $b\leq o_2$ for all $b\in A$.
     \item $u_1\leq u_2$ iff  $k\leq u_1$ implies $k\leq u_2$ for all $k\in K(A^\delta)$, iff $u_2\leq o$ implies $u_1\leq o$ for all $o\in O(A^\delta)$.
     \item 
     $k_1\vee k_2\in K(A^\delta)$ and  $o_1\wedge o_2\in O(A^\delta)$.
  \end{enumerate}   
 \end{proposition}

\begin{proposition}(cf.~\cite[Proposition 2.7]{de2024obligations2})
    \label{prop: compactness and existential ackermann}
 For any  lattice $A$, 
\begin{enumerate}[noitemsep,nolistsep]
   \item for any $b\in A$, 
   $k_1, k_2 \in K(A^\delta)$, and $o \in O(A^\delta)$, 
    \begin{enumerate}[label=(\roman*),noitemsep]
        \item $ k_1\wedge k_2 \leq b$ implies $a_1\wedge a_2\leq b$ for some $a_1, a_2 \in A$ s.t.~$ k_i \leq a_i$;
         \item $ k_1\wedge k_2 \leq o$ implies $a_1\wedge a_2\leq b$ for some $a_1, a_2 ,b \in A$ s.t.~$ k_i \leq a_i$ and $b\leq o$;
             \item 
             $\bigwedge K\in K(A^\delta)$ for every  $K\subseteq K(A^\delta)$.
    \end{enumerate}
    \item for any $a\in A$, 
   $o_1, o_2 \in O(A^\delta)$, and $k \in K(A^\delta)$, 
    \begin{enumerate}[label=(\roman*),noitemsep,nolistsep]
        \item $a\leq o_1\vee o_2 $ implies $a \leq b_1\vee b_2$ for some $b_1, b_2 \in A$ s.t.~$  b_i\leq o_i$;
         \item $ k \leq o_1\vee o_2 $ implies $a \leq b_1\vee b_2$ for some $a, b_1, b_2 \in A$ s.t.~$ b_i \leq o_i$ and $ k \leq a$.
         \item 
         $\bigvee O\in O(A^\delta)$ for every  $O\subseteq O(A^\delta)$.
    \end{enumerate}
    \end{enumerate}
\end{proposition} 
\noindent For the purposes of this paper, a {\em subordination algebra} is a tuple $\mathbb{S} = (A, \prec)$ such that $A$ is a distributive lattice and   ${\prec}\subseteq A\times A$  is a {\em subordination relation}, i.e.~$\prec$ satisfies the following conditions: for all $a, b,  c, d\in A$,

{\small
{{
\centering
\begin{tabular}{ll}
    ($\bot$-$\top$)  $\bot\prec\bot$ and $\top\prec\top$; &  (AND)  if $a \prec b$ and $a\prec c$ then $a \prec b \wedge c$;\\
     (OR)  if $a \prec c$ and $b\prec c$ then $a\vee b \prec  c$;   & (WO-SI)  if $a \leq b \prec c \leq d$ then $a \prec d$.
\end{tabular}
\par
}}
}

\section{Slanted Heyting algebras and subordination algebras} \label{sec:slanted Heyting}

Slanted Heyting algebras form a subclass of slanted DLE-algebras \cite[Definition 3.2]{de2021slanted}. 
\begin{definition}
    \label{def:slanted Heyting}
    A {\em slanted Heyting algebra} is a tuple $\mathbb{A} = (A, \si)$ s.t.~$A$ is a distributive lattice, and $\si: A\times A\to A^\delta$ s.t.~for all $a, b, c\in A$,
    \begin{enumerate}[noitemsep,nolistsep]
        \item $a\si b\in O(A^\delta)$;
        \item $a\si(b_1\wedge b_2) = (a\si b_1)\wedge (a\si b_2)$ and $a\si\top = \top$;
        \item $(a_1\vee a_2)\si b = (a_1\si b)\wedge (a_2\si b)$ and $\bot\si b = \top$;
        \item $c\leq a\si b$ iff $a\wedge c\leq \top \si b$.
    \end{enumerate}

The {\em canonical extension} of $\mathbb{A} = (A, \si)$ (cf.~\cite[Definition 3.4]{de2021slanted}) is $\mathbb{A}^\delta = (A^\delta, \si^\pi)$, where $A^\delta$ is the canonical extension of $A$,  and $\si^\pi: A^\delta \times A^\delta \to A^\delta$  is defined as follows:
 for every $k\in K(A^\delta)$, $o\in O(A^\delta)$ and $u,v\in A^\delta$,\
 
 {{\centering
 $k \si^\pi o : = \bigvee \{ a \si b   \mid  k \leq a, b \leq o, a,b \in A\}$
 \par
 }}

{{
\centering
$u \si^\pi v : = \bigwedge \{ k \si^\pi o \mid k \in K(A^\delta), o\in O(A^\delta), k \leq u, v \leq o\}$
\par
}}

\end{definition}
\noindent The map $\si^\pi$ extends $\si$, distributes over arbitrary meets in its second coordinate, and distributes arbitrary joins to  meets in its first coordinate (cf.~\cite[Lemma 3.5]{de2021slanted}). In what follows, we will omit the superscript $^\pi$, and rely on the arguments for disambiguation. Next, we discuss how slanted Heyting algebras can be understood as an equivalent presentation of subordination algebras. 
\begin{definition}
\label{def:upper lower star}
   The slanted Heyting algebra associated with the subordination algebra\footnote{\label{footnote}Notice that, by definition, $a\sip p = a\to\blacksquare b$, where $\blacksquare b \coloneqq \bigvee \prec^{-1}[b]$ and $\to$ is the Heyting algebra implication which is naturally defined on $A^\delta$ when $A$ is a distributive lattice. However, the definition as given and some of the ensuing proofs hold in a wider setting than that of distributive lattices. In particular,  
   Proposition \ref{prop:back and forth ast} holds verbatim if $\mathbb{S}$ is a general lattice-based proto-subordination algebra with ($\top$), ($\bot$), (SI) and (WO) and $\mathbb{A}$ is a slanted  algebra  s.t.~$\sip$ is antitone and $\bot$-reversing in the first coordinate and monotone and $\top$-preserving  in the second one. In such a setting, the equivalent characterization of $a\sip p $ as $a\to\blacksquare b$ is not available anymore, since $\to$ does not exist in general.} $\mathbb{S} = (A, \prec)$ is the tuple $\mathbb{S}_\ast\coloneqq (A, \sip)$, s.t.~for all $a, b\in A$,
   {\small
    \begin{equation}
        \label{eq:definitionsi of prec}a\sip b: = \bigvee\{c\in A\mid a\wedge c\prec b\}.\end{equation}}
     The subordination algebra associated with the slanted Heyting algebra $\mathbb{A} = (A, \si)$ is the tuple $\mathbb{A}^\ast \coloneqq(A, \precs)$ s.t.~for all $a, b\in A$,
     
     {{
     \centering
     $a\precs b$ iff $a\leq \top \si b$.
     \par
     }}
\end{definition}

\noindent If $\prec$ represents a normative system,  then  $a \sip b$ is the disjunction of all propositions that together with $a$ normatively imply $b$. That is, $a\wedge (a\sip b)\prec b$ always holds, and $a\sip b$  is the weakest element of $A^\delta$ with this property.

\begin{proposition}
\label{prop:back and forth ast}
    For any  subordination algebra $\mathbb{S} = (A, \prec)$, any slanted Heyting algebra $\mathbb{A} = (A, \si)$, and for all $a,b,c \in A$,
    \begin{enumerate}[noitemsep,nolistsep]
        \item   $c\leq a\sip b \quad \text{ iff } \quad a\wedge c\prec b$;
        \item $\mathbb{S}_\ast = (A, \sip)$ is a slanted Heyting algebra;
        \item $\mathbb{A}^\ast = (A, \precs)$ is a subordination algebra;
        \item  $a {\prec_{\si_{\prec}}}b$ iff $a\prec b$, and   $a\ {\si_{\prec_{\si}}}b = a\si b$.
    \end{enumerate}
\end{proposition}
\begin{proof}
(i) If $a\wedge c\prec b$, then  $c\leq \bigvee \{c\mid a\wedge c\prec b\} =  a\sip b$. If $c\leq a\sip b = \bigvee \{c\mid a\wedge c\prec b\}\in O(A^\delta)$, by compactness, $c\leq d$ for some $d\in A$ s.t.~$a\wedge d\prec b$. Hence, $a\wedge c\leq a\wedge d\prec b$, which implies $a\wedge c\prec b$ by (SI).

    (ii) By definition, $a\sip b = \bigvee \{c\mid a\wedge c\prec b\}\in O(A^\delta)$ for any $a, b\in A$. Moreover, $a\sip\top = \bigvee\{c\mid a\wedge c\prec \top\} = \top$, the last identity holding because properties $(\top)$ and (SI) hold for $\prec$. Likewise, $\bot\sip b = \bigvee\{c\mid \bot\wedge c\prec b\} = \top$, the last identity holding because properties $(\bot)$ and (WO) hold for $\prec$. Let $a, b_1, b_2\in A$, and let us show that
    
    {{
    \centering
    $a\sip(b_1\wedge b_2) = (a\sip b_1)\wedge (a\sip b_2)$.
    \par
    }}
    
   \noindent  For the left-to-right inequality, by Proposition \ref{prop:background can-ext} (ii) and item (i), 
     this is equivalent to show that, for any $c\in A$,
     
     {{
     \centering
     $a\wedge c\prec b_1\wedge b_2 $ iff $a\wedge c\prec b_1 \; \text{ and } \; a\wedge c\prec b_2$. 
     \par
     }}
    
    \noindent If $a\wedge c\prec b_1\wedge b_2$, then $a\wedge c\prec b_1\wedge b_2\leq b_i$ for $i = 1, 2$, which implies by (WO) that $a\wedge c\prec  b_i$, as required. Conversely, 
    if $a\wedge c\prec  b_i$ for $i = 1, 2$, then, by (AND), $a\wedge c\prec b_1\wedge b_2$, as required.
    Let $a_1, a_2, b\in A$, and let us show that 

    {{
    \centering
    $(a_1\vee a_2)\sip b = (a_1\sip b)\wedge (a_2\sip b)$.
    \par
    }}
    
\noindent    by Proposition \ref{prop:background can-ext} (ii) and  (i), this is equivalent to show that, for any $c\in A$,
    
    {{
    \centering
    $(a_1\vee a_2)\wedge c\prec b$ iff $a_1\wedge c\prec b \; \text{ and } \; a_2\wedge c\prec b$.
    \par
    }}

   \noindent If $(a_1\vee a_2)\wedge c\prec b$, then $a_i\wedge c\leq (a_1\vee a_2)\wedge c\prec b$,   which implies, by (SI), that  $a_i\wedge c\prec b$ for $i = 1, 2$, as required. Conversely, if $a_1\wedge c\prec b$ and $a_2\wedge c\prec b$, then by distributivity and  (OR), $(a_1\vee a_2)\wedge c = (a_1\wedge c)\vee (a_2\wedge c)\prec b$, as required.   
     Finally, let us show that 
    for all $a, b, c\in A$,
    
    {{
    \centering
    $a\wedge c\leq \top \sip b$ iff  $ c\leq a\sip b$.
    \par
    }}
    \noindent By item (i), it is enough to show that $(a\wedge c)\wedge \top\prec b$ iff $a\wedge c\prec b$, which is immediately true.

    (iii) As to $(\bot)$, $\bot\precs b$ iff $\bot\leq \top\si b$, which is clearly true. 
     As to $(\top)$, $a\precs \top$ iff $ a\leq \top\si \top$, which is  true by Definition \ref{def:slanted Heyting}.2. As to (WO), $a\precs b\leq b'$ implies $a\leq \top\si b\leq \top\si b'$, hence $a\precs b'$, as required. As to (SI), $a'\leq a\precs b$ implies $a'\leq a\leq \top\si b$, and hence $a'\precs b$, as required. As to (AND), if $a\precs b_i$ for $i = 1, 2$, then $a\leq \top\si b_i$, therefore $a\leq (\top\si b_1)\wedge (\top\si b_2) = \top\si (b_1\wedge b_2)$, and so $a\precs b_1\wedge b_2$, as required. As to (OR), if $a_i\leq b$ for $i = 1, 2$, then $a_i\leq \top\si b$, hence $a_1\vee a_2\leq \top\si b$ and so $a_1\vee a_2\precs b$, as required. 

     (iv) By Definition \ref{def:upper lower star} and item (i), $a {\prec_{\si_{\prec}}} b$ iff $a\leq \top \sip b $ iff $a = \top\wedge a\prec b$, as required. Finally, to show that $a\ {\si_{\prec_{\si}}}b = a\si b$, by Proposition \ref{prop:background can-ext} (ii), it is enough to show that for all $c\in A$,
     
     {{\centering
        $c\leq a\ {\si_{\prec_{\si}}}b$ iff  $c\leq a\si b$.
        \par
     }}
    \noindent By item (i), $c\leq a \ {\si_{\prec_{\si}}} b $ iff  $a\wedge c \precs b$ i.e.~$a\wedge c\leq \top\si b$, which, by Definition \ref{def:slanted Heyting}.4, is equivalent to $c\leq a\si b$, as required.
     \end{proof}

\begin{lemma}
\label{lemma:sip and prec}
    For any slanted Heyting algebra $\mathbb{A} = (A, \si)$, any $k,k' \in K(A^\delta)$,  $o \in O(A^\delta)$, and $u,v,w \in A^\delta$,
\begin{enumerate}[noitemsep,nolistsep]
    \item $k \si o \in  O(A^\delta)$;
    \item $
k\leq k'\si o\quad \text{ iff } \quad \exists a\exists b\exists c(a\leq c\si b \ \& \ k\leq a\ \&\ k'\leq c \ \&\   b\leq o)$;

\item $
k\land k'\leq\top\si o\quad \text{ iff } \quad \exists a\exists b\exists c(a\land c\leq\top\si b \ \& \ k\leq a\ \&\ k'\leq c \ \&\   b\leq o)$;
\item $k \leq k' \si o\quad $ iff $\quad k \wedge k' \leq \top \si o$;
\item $w \leq u \si v\quad $ iff $\quad w \wedge u \leq \top \si v$.
\end{enumerate}
\end{lemma}
\begin{proof}
    (i) By definition, $ k \si o = \bigvee \{ a \si b   \mid  k \leq a, b \leq o, a,b \in A\}$. This implies that $k\si o \in  O(A^\delta)$ by Proposition \ref{prop: compactness and existential ackermann} (ii), since $a \si b \in O(A^\delta)$  by Definition \ref{def:slanted Heyting} (i).

    {{
    \small
    \centering
    \begin{tabular}{ l c l l}
     (ii) && $k\leq k'\si o$ \\
        & iff & $k \leq \bigvee \{ c \si b \mid k' \leq c \ \&\ b \leq o \}$ & Def.  $\si^\pi$ \\
        & iff & $\exists a\exists b\exists c (a \leq c \si b \ \&\ k \leq a \ \&\ k' \leq c \ \&\ b \leq o)$ &  compactness. \\
    \end{tabular}
\par
}}

 {{
 \small
 \centering
    \begin{tabular}{l c l l}
     (iii) && $k\land k'\leq\top\si o$ \\
        & iff & $\bigwedge \{a\land c \mid k \leq a \ \&\ k' \leq c\} \leq \bigvee \{ \top \si b \mid  b \leq o \}$ & Def.  $\si^\pi$ \\
        & iff & $\exists a\exists b\exists c (a \land c\leq \top \si b \ \&\ k \leq a \ \&\ k' \leq c \ \&\ b \leq o)$ &  compactness. \\
    \end{tabular}
\par
}}

(iv) Immediate by items (ii) and (iii), and Definition \ref{def:slanted Heyting}.(iv).

(v) Since $A$ is distributive, $A^\delta$ is completely distributive (cf.~Section \ref{sec: preliminaries}), by denseness, $w \wedge u= \bigvee\{k \wedge k' \mid k  \leq w, k' \leq u\}$. Moreover,   $\top \si v=\bigwedge\{\top \si o \mid v \leq o\}$ and $u \si v=\bigwedge\{k \si o \mid k \leq u, v \leq o\}$ by definition of $\si^\pi$. Hence:

{{
\small
\centering
    \begin{tabular}{r c l l}
 && $w\leq u\si v$ \\
 &iff &$\bigvee\{k \mid k\leq w\}\leq \bigwedge\{k' \si o \mid k' \leq u, v \leq o\}$ \\
 &iff &$\forall k\forall k'\forall o(k\leq w\ \& \ k'\leq u\ \&\ v\leq o\implies  k\leq k' \si o)$ \\
 &iff &$\forall k\forall k'\forall o(k\leq w\ \& \ k'\leq u\ \&\ v\leq o\implies  k \wedge k' \leq \top \si o)$ & item (iv) \\
 &iff &$\bigvee\{k \wedge k' \mid k  \leq w, k' \leq u\}\leq\bigwedge\{\top \si o \mid v \leq o\}$ \\
 &iff &$w\land u\leq\top\si v$.
  \end{tabular}
\par
}}
\end{proof}

\section{Examples and discussion}
\label{sec:examples on sip}
In this section, we discuss  how the semantic environment of slanted Heyting 
algebras can be used to model different properties of normative systems in a similar style to the modal characterizations of \cite{de2024correspondence}. The properties of the previous section allow us to  characterize  well known conditions\footnote{Typically, the conditions we consider are expressed in terms of rules or Horn clauses, i.e.~conjunction of relational atoms entails  a relational atom, and the entailment relation will be represented by the symbol $\implies$.} of normative systems, such as (CT) and (T), in terms of axioms (i.e.~algebraic inequalities which represent sequents) involving slanted Heyting implications. For the sake of enhanced readability, in what follows, all non quantified variables  are quantified universally. 

{{
\small
\centering
    \begin{tabular}{cll}
(T) & $a \prec b \ \& \ b \prec c \implies a \prec c$ \\
iff & $a \leq \top\sip b \ \& \ b \leq \top\sip c \implies a \leq \top \sip c$ & Prop. \ref{prop:back and forth ast}(i)\\
iff & $\exists b(a \leq \top\sip b \ \& \ b \leq \top\sip c) \implies a \leq \top \sip c$ &   \\
iff & $a \leq \top\sip ( \top\sip c) \implies a \leq \top \sip c$ & Lemma \ref{lemma:sip and prec}(ii) \\
iff & $\top\sip ( \top\sip c)  \leq \top \sip c$ & Prop. \ref{prop:background can-ext}(ii) \\
 \end{tabular}    
\par
}}

{{
\small
\centering
\begin{tabular}{cll}
(CT) & $a \prec b \ \& \ a \wedge b \prec c \Rightarrow a \prec c$ \\
iff & $a \leq \top\sip b \ \& \ a \leq b \sip c \implies a \leq \top \sip c$ & Prop. \ref{prop:back and forth ast}(i) \\
iff & $a \leq (\top\sip b) \wedge (b \sip c) \implies a \leq \top \sip c$ \\
iff & $(\top\sip b) \wedge (b \sip c)  \leq \top \sip c$ & Prop. \ref{prop:background can-ext}(ii) \\
 \end{tabular}    
\par
}}

\noindent Conversely, we can  translate inequalities on slanted Heyting algebras into equivalent conditions on subordination algebras. For example, consider the following inequalities encoding transitivity (T$_{\sip}$) and cumulative transitivity (CT$_{\sip}$)  of the implication  $\sip$. 

{{
\small
\centering
    \begin{tabular}{cll}
   (T$_{\sip}$) & $(a\sip  b) \wedge (b\sip c)\leq a\sip c$\\
    iff  & $ d\leq (a\sip  b) \wedge (b\sip c)\implies d\leq a\sip c$ & Prop. \ref{prop:background can-ext}(ii)\\
    iff  & $d\leq a\sip  b \ \&\ d\leq b\sip c\implies d\leq a\sip c$\\
    iff  & $a\wedge d\prec  b \ \&\ b\wedge d\prec   c\implies a\wedge d\prec c$ & Prop. \ref{prop:back and forth ast}(i)\\
     
    \end{tabular}    
\par
}}

{{
\small
\centering
    \begin{tabular}{cll}
(CT$_{\sip}$)  & $(a \sip b) \wedge ((a \wedge b) \sip c) \leq a \sip c$ \\
 iff & $d\leq (a \sip b) \wedge ((a \wedge b) \sip c) \implies d\leq a \sip c$ & Prop. \ref{prop:background can-ext}(ii)\\
  iff & $d\leq a \sip b \ \&\ d\leq  (a \wedge b) \sip c \implies d\leq a \sip c$\\

     iff & $d\wedge a\prec b \ \&\ d\wedge (a \wedge b) \prec c \implies d\wedge a\prec c$ & Prop. \ref{prop:back and forth ast}(i)\\
 \end{tabular}    
\par
}}
\noindent Note that   (T) and (CT)  follow from the conditions equivalent to (T$_{\sip}$) and  (CT$_{\sip}$), respectively, by setting $d:=\top$.\footnote{The converse implication also holds in the case of (CT), by    substituting $a$ in (CT)   with $a \wedge d$ (recall that $a$ is universally quantified).} 
Hence,   (T$_{\sip}$) and  (CT$_{\sip}$) can be seen as  strengthening  (T) and (CT)  under any side condition or context  $d$.  

Other interesting conditions on normative systems can similarly be expressed in terms of the language of slanted Heyting algebras. 
For example, normative counterparts of intuitionistic tautologies such as  the {\em Frege axiom}:

{{
\small
\centering
    \begin{tabular}{lll}
    & $a\sip (b\sip c)\leq (a\sip b)\sip (a\sip c)$\\
     iff   & $k\leq a\sip (b\sip c)\ \& \ k'\leq a\sip b\implies k\leq k'\sip (a\sip c)$ &denseness\\
       iff   & $\exists d(k\leq d\leq a\sip (b\sip c))\ \& \ \exists e(k'\leq e\leq a\sip b)\implies k\leq k'\sip (a\sip c)$ &compactness\\
   
     iff   & $d\leq a\sip (b\sip c)\ \& \ e\leq a\sip b\implies d\leq e\sip (a\sip c)$& ($\ast$)\\
      iff   & $d\leq a\sip f\ \& \ f\leq  b\sip c\ \& \ e\leq a\sip b\implies \exists g(d\leq e\sip g\ \&\ g\leq a\sip c)$ & compactness\\
      iff   & $a\wedge d\prec f\ \& \ a\wedge e\prec b\ \& \ b\wedge f\prec c\ \implies \exists g(e\wedge d\prec g\ \&\ a\wedge g\prec c)$ & Prop. \ref{prop:back and forth ast}(i)\\
    \end{tabular}    
\par
}}

\noindent As to the equivalence marked $(\ast)$, from bottom to top, fix $a, b, c\in A$ and $k, k'\in K(A^\delta)$ s.t.~$\exists d(k\leq d\leq a\sip (b\sip c))\ \& \ \exists e(k'\leq e\leq a\sip b)$. Then $d\leq a\sip (b\sip c)\ \& \ e\leq a\sip b$, hence by assumption, $k\leq d\leq e\sip (a\sip c)\leq k\leq k'\sip (a\sip c)$. From top to bottom, it is enough to instantiate $k: = d$ and $k': = e$.

 The condition above can be understood as a mode of transitive propagation of obligations under context: if $a$ normatively implies $f$ whenever $d$, and $b$ whenever $e$, and $b$ and $f$ together normatively imply $c$, then $c$ is also normatively implied by $a$ in the context of some $g$  that is normatively implied by $d\wedge e$. This condition can be understood as a generalization of the following principle:

{{
\small
    \begin{tabular}{lll}
 & $a\wedge d\prec f\ \& \ a\wedge e\prec b\ \& \ b\wedge f\prec c\ \implies  a\wedge (d\wedge e)\prec c$\\
 iff  & $a\leq d\sip f\ \& \ a\leq e\sip b\ \& \ b\leq  f\sip c\ \implies  a\leq (d\wedge e)\sip c$\\
 iff  & $a\leq d\sip f\ \& \ \exists b(a\leq e\sip b\ \& \ b\leq  f\sip c)\ \implies  a\leq (d\wedge e)\sip c$\\
  iff  & $a\leq d\sip f\ \& \ a\leq e\sip ( f\sip c)\ \implies  a\leq (d\wedge e)\sip c$\\
    iff  & $a\leq (d\sip f)\wedge ( e\sip ( f\sip c))\ \implies  a\leq (d\wedge e)\sip c$\\
       iff  & $ (d\sip f)\wedge ( e\sip ( f\sip c))\leq (d\wedge e)\sip c$\\
    \end{tabular}    
\par
}}

\noindent Interesting conditions can be also captured in terms of the normative counterparts of axioms defining intermediate logics, such as  the normative counterpart of the  {\em G\"odel-Dummett axiom}: 

{{
\small
\centering
    \begin{tabular}{cll}
  & $ \top \leq (a \sip b) \vee (b \sip a)$ \\
  iff & $  \exists e \exists f (\top \leq e \vee f \ \&\  e \leq  a \sip b \ \&\ f\leq b \sip a)$ & Prop. \ref{prop: compactness and existential ackermann}(ii)\\
     iff & $  \exists e \exists f (\top \leq e \vee f \ \&\  a \wedge e \prec b \ \&\  b \wedge f \prec a)$ & Prop. \ref{prop:back and forth ast}(i)\\
 \end{tabular}    
\par
}}
\noindent The  condition  above  requires there to be propositions $e$ and $f$ such that $e$ or $f$ is always the case and $a$ and $e$ normatively imply $b$ and  $b$ and $f$  normatively imply $a$.  This requirement can be seen as a generalization of the following dichotomy axiom for a normative system: $\forall a\forall  b(a \prec b \text{ or } b \prec a )$, i.e.~for all propositions $a$ and $b$, either   $a$ normatively implies $b$ or  $b$ normatively implies $a$.  

The following axiom encodes the distributivity of $\sip$  over disjunction in its second coordinate. 

{{
\small
\centering
    \begin{tabular}{cll}
  & $a \sip (b \vee c) \leq (a \sip b) \vee (a \sip c)$\\
iff & $d \leq a \sip (b \vee c) \implies d \leq  (a \sip b )\vee (a \sip c)$ & Prop. \ref{prop:background can-ext}(ii)\\ 
  iff & $ d \wedge a \prec b \vee c \implies  \exists e \exists f (d \leq e \vee f \, \& \,  e \leq a \sip b \, \& \, f \leq a \sip c )$ & Prop. \ref{prop: compactness and existential ackermann}(ii)\\
 iff & $  d \wedge  a \prec b \vee c \implies \exists e \exists f (d \leq e \vee f \ \& \ e \wedge a \prec b \ \& \ a \wedge f \prec c  )$ & Prop. \ref{prop:back and forth ast}(i)\\
 \end{tabular}    
\par
}}
\noindent This axiom characterizes a form of splitting into cases for  conditional obligations: if $a$
and $d$ normatively imply $b \vee c$ then some $e$ and $f$ exist s.t.~$d$ implies $e \vee f$ and $a$ and $e$ normatively imply  $b$ while $a$ and $f$ normatively imply  $c$. 

\section{Slanted co-Heyting algebras}
\label{sec: slanted co-heyting}
     \begin{definition}
    \label{def:slanted co-Heyting}
    A {\em slanted co-Heyting algebra} is a tuple $\mathbb{A} = (A, \sci)$ s.t.~$A$ is a distributive lattice, and $\sci: A\times A\to A^\delta$ s.t.~for all $a, b, c\in A$,
    \begin{enumerate}[nolistsep, noitemsep]
        \item $a\sci b\in K(A^\delta)$;
        \item $a\sci(b_1\vee b_2) = (a\sci b_1)\vee (a\sci b_2)$ and $a\sci\bot = \bot$;
        \item $(a_1\wedge a_2)\sci b = (a_1\sci b)\vee (a_2\sci b)$ and $\top\sci b = \bot$;
        \item $ a\sci b\leq c$ iff $\bot\sci b\leq a\vee c$.
    \end{enumerate}
The {\em canonical extension} of $\mathbb{A}$ is $\mathbb{A}^\delta = (A^\delta, \sci^\sigma)$, where $A^\delta$ is the canonical extension of $A$,  and $\sci^\sigma: A^\delta \times A^\delta \to A^\delta$  is defined as follows:
 for every $k\in K(A^\delta)$, $o\in O(A^\delta)$ and $u,v\in A^\delta$,
 
 {{
 \centering
 $ o \sci^\sigma  k: = \bigwedge\{ a \sci b   \mid  a \leq o, k \leq b, a,b \in A\}$
 \par
 }}

 {{
 \centering
 $u \sci^\sigma v : = \bigwedge \{ k \si^\sigma o \mid k \in K(A^\delta), o\in O(A^\delta), k \leq u, v \leq o\}$
 \par
 }}

\end{definition}
\noindent The map $\sci^\sigma$ extends $\sci$, distributes over arbitrary joins in its second coordinate, and distributes arbitrary meets to  joins in its first coordinate (cf.~\cite[Lemma 3.5]{de2021slanted}). In what follows, we will omit the superscript $^\sigma$, and rely on the arguments for disambiguation. Slanted co-Heyting algebras are also an equivalent presentation of subordination algebras.
\begin{definition}
\label{def:upper lower bullet}
  The slanted co-Heyting algebra associated with a subordination algebra $\mathbb{S} = (A, \prec)$ is the tuple $\mathbb{S}_\bullet\coloneqq (A, \scip)$, s.t.~for all $a, b\in S$,
  {\small  \begin{equation}
        \label{eq:definitionsci of prec}
    a\scip b: = \bigwedge\{c\in A\mid b\prec a\vee c\}.\end{equation}}
\noindent Hence,   $a \scip b$ represents the conjunction of all propositions  whose disjunction with $a$  is
 normatively implied by $b$. That is, $a\prec (a\scip b) \vee  b$, and $a\scip b$ holds, and is the strongest element of $A^\delta$ with this property. The subordination algebra associated with  $\mathbb{A} = (A, \sci)$ is  $\mathbb{A}^\bullet \coloneqq(A, \prec_{\sci})$ s.t.~for all $a, b\in A$,

{{
\centering
$a \prec_{\sci} b$ iff $\bot \sci a\leq b$. 
\par
}}
\end{definition}

The following proposition is dual to Proposition \ref{prop:back and forth ast}. 
\begin{proposition}
\label{prop:back and forth bullet}
    For any  subordination algebra $\mathbb{S} = (A, \prec)$, any slanted co-Heyting algebra $\mathbb{A} = (A, \sci)$, and all $a,b,c \in A$,  
    \begin{enumerate}[noitemsep,nolistsep]
        \item $ a\scip b  \leq c \quad \text{ iff } \quad b \prec a \vee c$;
        \item $\mathbb{S}_\bullet = (A, \scip)$ is a slanted co-Heyting algebra;
        \item $\mathbb{A}^\bullet = (A, \prec_{\sci})$ is a subordination algebra;
        \item $a {\prec_{\scip}}b$ iff $a\prec b$, and   $a\ {\sci_{\prec_{\sci}}}b = a\sci b$.
    \end{enumerate}
\end{proposition}

The following lemma is dual to Lemma  \ref{lemma:sip and prec}.

\begin{lemma}
\label{lemma:scip and prec}
    For any slanted co-Heyting algebra $\mathbb{A} = (A, \sci)$, any $o,o' \in O(A^\delta)$,  $k \in K(A^\delta)$, and $u,v,w \in A^\delta$,
\begin{enumerate}[noitemsep,nolistsep]
    \item $k \sci o \in  K(A^\delta)$;
    \item $
o'\sci k\leq o\quad \text{ iff } \quad \exists a\exists b\exists c(a\sci b\leq c \ \& \  a\leq o'\ \&\ k\leq b \ \&\   c\leq o)$;

\item $
\bot \sci k\leq o'\vee o\quad \text{ iff } \quad \exists a\exists b\exists c(\bot \sci b\leq a\vee c \ \& \  a\leq o'\ \&\ k\leq b \ \&\   c\leq o)$;
\item $o'\sci k\leq o\quad $ iff $\quad \bot\sci k\leq o\vee o'$;
\item $w \sci u \leq v\quad $ iff $\quad \bot \sci u \leq w \vee v$.
\end{enumerate}
\end{lemma}

\section{Slanted pseudo-complements and co-complements}\label{sec:slanted negations}
Any slanted (co-)Heyting algebra based on a distributive lattice $A$ induces the  operation $\sn: A\to A^\delta$ (resp.~$\scn: A\to A^\delta$) defined by the assignment 
 $a\mapsto a\si \bot$ (resp.~$a \mapsto \top \sci a $). When ${\si} = {\sip}$ (resp.~${\sci} = {\scip}$) for some  $\prec$ on $A$, it immediately follows from \eqref{eq:definitionsi of prec} and \eqref{eq:definitionsci of prec} that, for any $a\in A$,

{{
\centering
$\sn a = \bigvee\{c\in A\mid  a\wedge c\prec \bot\}\quad \quad  \scn a = \bigwedge\{c\in A\mid \top\prec a\vee c\}$.
\par
}}

\noindent In particular,

{{
\centering
$\label{eq:neg top and sim bot}
  \sn \top = \bigvee\{c\in A\mid  c\prec \bot\}\quad \quad  \scn \bot = \bigwedge\{c\in A\mid \top\prec  c\}$.
\par
}}

The canonical extensions $\sn^\pi$ and $\scn^\sigma$ of maps $\sn$ and $\scn$ are
defined as follows: for any $k \in K(A^\delta)$, $o \in O(A^\delta)$, $u \in A^\delta$, 

{{\centering
$\sn^\pi k : = \bigvee\{ \sn a \mid k \leq a, a \in A\}$ and $\sn^\pi  u : = \bigwedge \{\sn k  \mid k \leq u\}$
\par
}}

{{\centering
$\scn^\sigma o : = \bigwedge\{ \scn a \mid a \leq o, a \in A\}$ and $\scn^\sigma  u : = \bigvee \{\scn o  \mid u \leq o\}$
\par
}}

\noindent We will typically omit the superscripts $^\sigma$ and $^\pi$, and rely on the arguments for disambiguation. Also, we omit the subscript $_\prec$ even when $\neg$ and $\sim$ arise from $\sip$ and $\scip$.
The next lemma is straightforward from the definitions. 

\begin{lemma}
    For any 
    $u \in A^\delta$, $\sn u = u \si \bot$ and 
   $\scn u = u  \sci \top$.
\end{lemma}

 \noindent Intuitively, $\sn a $  represents the disjunction of all propositions which are normatively inconsistent with $a$. That is, $a \wedge \sn a \prec \bot$ and $\sn a $  is the weakest element of $A^\delta$ with this property. Likewise,  $\scn a$ represents the conjunction  of all propositions whose disjunction with  $a$  is an unconditional obligation. That is, $\top \prec a \vee \scn a$ and $\scn a $  is the strongest element of $A^\delta$ with this property. In particular,  $\sn\top$ denotes the weakest normatively inconsistent element   in $A^\delta$, while $\scn\bot$ denotes the strongest unconditional obligation  in $A^\delta$. 

\noindent The following lemma is a straightforward consequence of Lemmas \ref{lemma:sip and prec} and \ref{lemma:scip and prec}.
\begin{lemma}\label{lem:negation properties}
    For all $a, c\in A$, and $k\in K(A^\delta)$ and $o\in O(A^\delta)$,
    \begin{enumerate}[noitemsep,nolistsep]
        \item $c\leq \sn a$ iff $a\wedge c\prec \bot\quad $ and 
        $\quad\scn a\leq c$ iff $\top\prec c\vee a$; 
        \item 
        $ a \leq \sn b$ iff $b \leq \sn a\quad$ and 
        $\quad\scn a \leq b$ iff $\scn b \leq a$;
         \item $k'\leq \sn k$ iff $\exists a\exists c (k'\leq c\ \&\ k\leq a\ \&\ c\leq \sn a)$;
        \item $\scn o\leq o'$ iff $\exists a\exists c(a\leq o\ \&\ c\leq o'\ \&\ \scn a\leq c)$.
        
    \end{enumerate}
\end{lemma}

\section{More examples}
\label{sec:more examples}
Condition  $a \prec \bot \implies a \leq \bot$  can be understood as the property of the normative system $\prec$ that  any consistent (proposition)  cannot  yield a normative inconsistency. This condition can be axiomatically captured as follows

{{
\small
\centering

\begin{tabular}{cll}

 & $a \prec \bot \implies a \leq \bot$ & \\
 iff & $a \leq \neg \top   \implies a \leq \bot$ & Lemma \ref{lem:negation properties}(i)\\
 iff & $\sn \top \leq \bot$ & Prop. \ref{prop:background can-ext} (ii)\\
 \end{tabular}  
\par
}}
\noindent The same condition can be also captured 
by  $a \wedge \sn a  \leq \bot$, as shown by the following computation:

{{\centering
\small
\begin{tabular}{cll}
& $a \wedge \sn a  \leq \bot$ & \\
 iff & $b  \leq a \wedge \neg a   \implies b \leq \bot$  & Prop. \ref{prop:background can-ext} (ii) \\
 iff &  $b  \leq a \, \& \,  b \leq \neg a    \implies b \leq \bot$ &   \\
 iff  &  $b  \leq a \, \& \,  b \wedge a  \prec \bot   \implies b \leq \bot$ &  Lemma \ref{lem:negation properties}(i)\\ 
 iff  &  $  b  \prec \bot   \implies b \leq \bot$ & $b\leq a$ iff $b = b\wedge a$ \\
 \end{tabular}    
\par
}}


\noindent The next example is normative counterpart of De Morgan axiom: 

{{
\small
\centering
\begin{tabular}{cll}
& $\sn (a \wedge b) \leq \sn a \vee \sn b $ & \\

 iff &  $d \leq \sn (a \wedge b) \implies d \leq \sn a \vee \sn b $   & Prop. \ref{prop:background can-ext} (ii) \\

 iff &  $d \leq \sn (a \wedge b) \implies \exists e \exists f (d \leq e \vee f \, \&, \,
 e \leq \sn a \, \& \, f \leq  \sn b)$   &  Prop. \ref{prop: compactness and existential ackermann}(ii)\\
 
 iff &  $d \wedge a \wedge b \prec \bot \implies \exists e \exists f (  d \leq e \vee f \, \&, \,
 e \wedge  a  \prec \bot \, \& \, f \wedge b \prec \bot)$   & Lemma \ref{lem:negation properties} (i)\\
 \end{tabular}    
\par
}}

\noindent The condition above  can be interpreted as saying that if $a, b, d$ are normatively inconsistent, then $d$ implies $e \vee f$ s.t. $e$ leads to normative inconsistency along with $a$ and $f$ leads  to normative inconsistency along with $b$. 

The next example is the normative counterpart of contraposition:

{{
\centering
\small
         \begin{tabular}{cll}
&  $(a \sip b) \leq (\neg b \sip \neg a) $& \\
 iff & $k\leq (a \sip b)\ \& \ j\leq \neg b \implies  (k \leq j \sip \neg a)$& denseness\\
  iff & $c\leq (a \sip b)\ \& \ d\leq \neg b \implies  (c \leq d \sip \neg a)$ &compactness\\
    iff & $c\wedge a \prec b\ \& \ d\wedge b\prec \bot \implies  (c \leq d \sip  \neg a)$  &Prop. \ref{prop:back and forth ast}, Lem. \ref{lem:negation properties}(i) \\ 
           iff & $c\wedge a \prec b\ \& \ d\wedge b\prec \bot \implies \exists e (c \leq d \sip e\ \&\ e\leq \neg a)$  &compactness \\
            iff & $c\wedge a \prec b\ \& \ d\wedge b\prec \bot \implies \exists e (c \wedge d \prec e\ \&\ e\wedge a\prec \bot)$  & Lem. \ref{lem:negation properties}(i)\\
         \end{tabular}    
\par
}}

\noindent The condition above says that, if for all $a$, $c$ and $d$ some $b$ exists  s.t.~$c\wedge a$  normatively imply $b$, and $b\wedge d$  are normatively inconsistent, then some $e$ exists  s.t.~$c\wedge d$  normatively imply $e$ and $e$ and $a$ together give a normative inconsistency. This condition generalizes the following principle:

{{
\small
\centering
           \begin{tabular}{cll}
                &   $a \prec b \, \& \, d \wedge b \prec \bot \implies \exists e (d \prec e  \, \& \, a \wedge e \prec \bot)$ \\
        iff  & $ \exists b (a \leq  \top  \sip b \, \& \, b \leq \sn d) \implies \exists e (d \leq  \top  \sip e \, \& \, e \leq \sn a )$& def. of $\sn$\\
        iff  & $a \leq  \top  \sip \sn d  \implies d \leq \top \sip \sn a$& denseness\\
        iff  & $a \leq  \top  \sip \sn d  \implies d \circ \top \leq \sn a$& residuation\\
        iff  & $a \leq  \top  \sip \sn d  \implies a \leq \sn (d \circ \top) $& Lemma \ref{lem:negation properties} (ii)\\
         iff  & $ \top  \sip \sn d   \leq \sn (d \circ \top) $ &Prop. \ref{prop:background can-ext}(ii)\\
         iff & $d \circ \top \leq \sn (\top  \sip \sn d)$&Lemma \ref{lem:negation properties} (ii)\\
          iff & $d \leq \top \sip \sn (\top  \sip \sn d)$& residuation\\
           \end{tabular}
 \par
 }}
\noindent where, in the computation above, $\circ$ denotes the left residual of $\sip^\pi$, which is defined as follows: $u\circ v = \bigwedge \{w\mid v\leq u\sip w \}$ for all $u, v, w\in A^\delta$.
%
%
The condition above can be interpreted as the requirement that, for all $a$ and $d$, if $a$ normatively implies some $b$ which is normatively inconsistent with $d$, then $d$ normatively implies some $e$ which is normatively inconsistent with $a$.

\section{Correspondence and inverse correspondence}\label{sec: correspondence inverse corr}

The examples discussed in Sections \ref{sec:examples on sip} and \ref{sec:more examples} are not isolated cases: in \cite{de2024correspondence}, the  class of  {\em (clopen-)analytic} axioms/inequalities is identified, each  of which is shown to be equivalent to some condition on norms, permissions, or their interaction. Conversely, a class of such conditions is identified (referred to as {\em Kracht formulas}), each  of which  is shown to be equivalent to some modal axioms. In this section, we discuss how  these results can be extended to the language  of (distributive) lattices with slanted (co-)Heyting implications,\footnote{Since the current signature is particularly simple, the definition of clopen-analytic inequality collapses to that of analytic inequality.} and show, via examples, that this language allows us to (algorithmically) capture conditions which cannot be captured by modal axioms of the language of \cite{de2024correspondence} with the same techniques. 
 
\subsection{Correspondence}
Let $\mathcal{L}$ be the language  of distributive lattices expanded with one slanted Heyting (co-)implication. To characterize syntactically the class of $\mathcal{L}$-inequalities which are guaranteed to be algorithmically transformed into conditions on normative systems, we adapt the notion of {\em analytic $\mathcal{L}$-inequalities} of \cite[Section 2.5]{de2024correspondence}. 

The \emph{positive} (resp.~\emph{negative}) {\em generation tree} of any $\mathcal{L}$-term $\phi$ is defined by labelling the root node of the generation tree  of $\phi$ with the sign $+$ (resp.~$-$), and then propagating the labelling on each remaining node as follows:
        For any node labelled with $ \lor$ or $\land$, assign the same sign to its children nodes, and 
        for any node labelled with $\si$ or $\sci$, assign the opposite sign to its first child node, and the same sign to its second child node.
	Nodes in signed generation trees are \emph{positive} (resp.~\emph{negative}) if they are signed $+$ (resp.~$-$). In the context of term inequalities $\varphi\leq \psi$, we consider the positive generation tree $+\varphi$ for the left side and the negative one $-\psi$ for the right side.
Non-leaf nodes in signed generation trees are called \emph{$\Delta$-adjoints}, \emph{syntactically left residuals (SLR)}, \emph{syntactically right residuals (SRR)}, and \emph{syntactically right adjoints (SRA)}, according to the specification given in the table below 
	Nodes that are either classified as  $\Delta$-adjoints or SLR are collectively referred to as {\em Skeleton-nodes}, while SRA- and SRR-nodes are referred to as {\em PIA-nodes}. 
	A branch in a signed generation tree $\ast \phi$, with $\ast \in \{+, - \}$, is  \emph{good} if it is the concatenation of two paths $P_1$ and $P_2$, one of which may possibly be of length $0$, such that $P_1$ is a path from the leaf consisting (apart from variable nodes) only of PIA-nodes, and $P_2$ consists (apart from variable nodes) only of Skeleton-nodes. An $\mathcal{L}$-inequality $\varphi \leq \psi$ is \emph{analytic} if every branch of $+\varphi$ and $-\psi$ is good.

    \begin{table}[h]
    {\footnotesize
		\begin{center}
			\bgroup
			\def\arraystretch{1.2}
			\begin{tabular}{| c | c |}
            
				\hline
				Skeleton &PIA\\
				\hline
				$\Delta$-adjoints & Syntactically Right Adjoint (SRA) \\
				\begin{tabular}{ c c c c}
					$+$ & & $\vee$ & \\
					$-$ & & $\wedge$ \\
				\end{tabular}
				&
				\begin{tabular}{c c c c }
					$+$& &$\wedge$ &  $\sn$ \\
					$-$ & &$\vee$ & $\scn$ \\

				\end{tabular}
				\\ \hline
				Syntactically Left Residual (SLR) & Syntactically Right Residual (SRR) \\
				\begin{tabular}{c c c c c}
					$+$ &  &$ \wedge$ & $\sci $    & $\sn$ \\
					$-$ &  &$\vee$ & $\si$ & $\scn$  \\
				\end{tabular}
				&\begin{tabular}{c c c c c }
				$+$ &  &$ \vee$ & $\si$ \\	
                    $-$ &  &$ \wedge$& $\sci$ 		
				\end{tabular}
				\\
				\hline
			\end{tabular}
			\egroup
		\end{center}
        \label{Join:and:Meet:Friendly:Table}
		\vspace{-1em}
    }
	\end{table}
    
%
%
 Based on the properties discussed in Sections \ref{sec:slanted Heyting} and \ref{sec: slanted co-heyting}, the algorithm of \cite[Section 3]{de2024correspondence} can succesfully be run also on analytic $\mathcal{L}$-inequalities; hence, the analogue of \cite[Theorem 3.1]{de2024correspondence} holds for analytic $\mathcal{L}$-inequalities. All inequalities discussed  in Sections \ref{sec:examples on sip} and \ref{sec:more examples} are analytic $\mathcal{L}$-inequalities, and the chains of equivalences discussed there represent runs of the correspondence algorithm.

\subsection{Inverse correspondence}

The following abbreviations will be used throughout the present section:

{\footnotesize
{{\centering
\begin{tabular}{lclcl}
$(\exists y \succ x)\varphi$ & $\equiv$ & $\exists y(x \prec y \ \& \ \varphi)$ & 
i.e.& $\exists y(x \leq \top\sip y \ \& \ \varphi)$ \\
$(\exists y \prec x)\varphi$ & $\equiv$ & $\exists y(y \prec x \ \& \ \varphi)$ & 
i.e. &$\exists y(\bot\scip  y \leq x \ \& \ \varphi)$ \\
$(\forall y \succ x)\varphi$ & $\equiv$ & $\forall y(x \prec y \ \implies \ \varphi)$ & 
i.e.& $\forall y(x \leq \top\sip y \ \implies \ \varphi)$ \\
$(\forall y \prec x)\varphi$ & $\equiv$ & $\forall y(y \prec x \ \implies \ \varphi)$ & 
i.e. &$\forall y(\bot\scip y \leq x \ \implies \ \varphi)$ \\
$(\exists \overline y \leq_\vee x)\varphi$ &
$\equiv$ &
$\exists y_1\exists y_2(x \leq y_1 \vee y_2 \ \& \ \varphi)$ \\
$(\exists \overline y \leq_\wedge x)\varphi$ &
$\equiv$ &
$
\exists y_1\exists y_2(y_1 \wedge y_2 \leq x \ \& \ \varphi) $\\
$(\forall \overline y \leq_\vee x)\varphi$ &
$\equiv$ &
$
\forall y_1\forall y_2(x \leq y_1 \vee y_2 \Longrightarrow \varphi)$ \\

$(\forall \overline y \leq_\wedge x)\varphi $ &
$\equiv$ &
$
\forall y_1\forall y_2(y_1 \wedge y_2 \leq x \Longrightarrow \varphi)$. \\
\end{tabular}

}}
{{\centering
\begin{tabular}{lclcl}
$(\exists \overline{y} \leq_{\scip} x)\varphi$ &  $\equiv$  & $\exists  y_1\exists y_2(y_2\prec y_1 \vee x \ \& \ \varphi)$ & i.e. & $\exists  y_1\exists y_2(y_1\scip y_2 \leq x \ \& \ \varphi)$\\
$(\exists \overline{y} \geq_{\sip} x)\varphi$ & $\equiv$ & $\exists y_1\exists y_2(x \wedge y_1\prec  y_2 \ \& \ \varphi)$ & i.e. & $\exists y_1\exists y_2(x \leq y_1\sip  y_2 \ \& \ \varphi)$\\
$(\forall \overline{y} \leq_{\scip} x)\varphi$ &  $\equiv$  & $\forall  y_1\forall y_2(y_2\prec y_1 \vee x \implies \varphi)$ & i.e. & $\forall  y_1\forall y_2(y_1\scip y_2 \leq x \implies \varphi)$\\
$(\forall \overline{y} \geq_{\sip} x)\varphi$ & $\equiv$ & $\forall y_1\forall y_2(x \wedge y_1\prec  y_2 \implies \varphi)$ & i.e. & $\forall y_1\forall y_2(x \leq y_1\sip  y_2 \implies \varphi)$
\end{tabular}
}}
}
Expressions  such as $(\forall y\prec x)$ or $(\exists\overline y\geq_g x)$ above are referred to as {\em restricted quantifiers}.
The variable $y$ (resp.~the variables in $\overline y$)  in the formulas above is (resp.\ are) {\em restricted}, and the variable $x$ is {\em restricting}, while the inequality occurring together with $\varphi$ in the translation of the restricted quantifier is a {\em restricting inequality}.
Throughout this section, we use the following letters to distinguish the roles of different variables ranging in the domain of an arbitrary slanted (co-)Heyting algebra.  The different conditions assigned to these variables in Definition \ref{def:inverse_shape} will determine how they are introduced/eliminated:

\begin{itemize}[noitemsep,nolistsep]
\item[$v$] variables occurring in the algebraic axiom;
\item[$a$] positive variables introduced/eliminated using Proposition \ref{prop:background can-ext}(ii). It must be possible to rewrite the quasi-inequality so that each such variable  occurs only once on each side of the main implication;
\item[$b$] negative variables introduced/eliminated using Proposition \ref{prop:background can-ext}(i). The same considerations which apply to $a$-variables apply also to $b$-variables;
\item[$c$] variables introduced/eliminated  using Proposition \ref{prop: compactness and existential ackermann} or Lemmas \ref{lemma:sip and prec},  \ref{lemma:scip and prec},  \ref{lem:negation properties}. in the antecedent of the main implication;
\item[$d$] variables introduced/eliminated in the same way as $c$-variables in the consequent of the main implication.
\end{itemize}

\noindent The following definition adapts \cite[Definition 5.2]{de2024correspondence} to the present environment.
\begin{definition}
\label{def:inverse_shape}
A  {\em Kracht formula}\footnote{In the modal logic literature, Kracht formulas (cf.~\cite[Section 3.7]{blackburn2001modal}, \cite{kracht1999tools}) are  sentences in the first order language of Kripke frames which are (equivalent to) the first-order correspondents of  Sahlqvist axioms. This notion has been generalized in \cite{palmigiano2024unified,conradie2022unified} from classical modal logic to (distributive) LE-logics, and from a class of first order formulas targeting Sahlqvist LE-axioms to a class targeting the strictly larger class of inductive LE-axioms. The notion of Kracht formulas introduced in Definition \ref{def:inverse_shape} is different and in fact incomparable with those in \cite{palmigiano2024unified,conradie2022unified}, since it targets a different and incomparable class of modal axioms.} is a condition of the following shape:

{{
\centering
$\forall \overline v\forall \overline a, \overline b 
(\forall \overline{c_m}R'_m z_m)\cdots(\forall \overline{c_1} R'_1 z_1)
\left(
\eta
\Rightarrow
(\exists \overline{d_o} R_o y_o)\cdots(\exists \overline{d_1} R_1 y_1)
\zeta
\right)$, where
\par
}}

\begin{enumerate}[noitemsep,nolistsep]
\item $R'_i, R_j\in \{\leq, \geq, \prec, \succ, \leq_{\scip}, \geq_{\sip} \}$ for all $1\leq i\leq m$ and $1\leq j\leq o$;
\item variables in $\overline z$ are amongst those of $\overline a$, $\overline b$, and $\overline c $; 
variables in $\overline y$ are amongst those of $\overline a$, $\overline b$, and $\overline d$;
\item $\eta$ and $\zeta$ are conjunctions of relational atoms $sRt$ with $R\in \{ \leq, \geq, \prec, \succ \}$;
\item all occurrences of variables in $\overline a$ (resp.\ $\overline b$) 
are positive (resp.\ negative) in all atoms (including those in restricting quantifiers) in which they occur;

\item every variable in $\overline c$ (resp.\ in $\overline d$) occurs uniformly in $\eta$ (resp.~in $\zeta$);

\item occurrences of variables in $\overline c$ (resp.\ $\overline d$) as restricting variables have the same polarity as their occurrences in $\eta$ (resp.\ $\zeta$).

\item all occurrences of variables in $\overline a$, $\overline b$, $\overline c$, and $\overline d$ in atoms of  $\eta$ and $\zeta$ are displayable;

\item each atom in $\eta$ contains exactly one occurrence of a variable in $\overline a$, $\overline b$, or $\overline c$ (all other variables are in $\overline v$);

\item each atom $sRt$ in $\zeta$ contains at most one occurrence of a variable in $\overline d$. Moreover, for every two different occurrences of the same variable not in $\overline v$ (i.e., in $\overline a$ or $\overline b$), the first common ancestor in the signed generation tree of $sRt$ is either $+\wedge$ or $-\vee$.
\end{enumerate}
\end{definition}

For instance, the following condition (see last example of Section \ref{sec:examples on sip}):

{{
\centering
$  d \wedge  a \prec b \vee c \implies \exists e \exists f (d \leq e \vee f \ \& \ e \wedge a \prec b \ \& \ a \wedge f \prec c  )$
\par
}}

can be recognized as an instance of the  Kracht shape in the language of slanted Heyting algebras by assigning variables $e$ and $f$ the role of $\overline{d}$-variables, variable $d$ the role of $\overline{a}$-variable, and variables $a, b, c$ the role of $\overline{v}$-variables, and moreover, letting $d\leq e\vee f$ be the restricting inequality in the consequent, and $\eta: = d \wedge  a \prec b \vee c $ and $\zeta: = e \wedge a \prec b \ \& \ a \wedge f \prec c$.
The algorithm introduced in \cite[Section 6]{de2024correspondence} allows us to equivalently represent this condition as an axiom in the language of slanted Heyting algebras as follows:
{\footnotesize
{{
\centering
    \begin{tabular}{cll}
    & $  d \wedge  a \prec b \vee c \implies \exists e \exists f (d \leq e \vee f \ \& \ e \wedge a \prec b \ \& \ a \wedge f \prec c  )$ & \\
    iff & $d\leq a\sip (b\vee c)\implies \exists e \exists f (d \leq e \vee f \ \& \ e \leq a \sip b \ \& \ f \leq a \sip c  ) $ & Prop. \ref{prop:back and forth ast}(i)\\
    iff & $d \leq a \sip (b \vee c) \implies d \leq  (a \sip b )\vee (a \sip c)$ & compactness \\ 
iff   & $a \sip (b \vee c) \leq (a \sip b) \vee (a \sip c)$ &  Prop. \ref{prop:background can-ext}(ii)\\
 \end{tabular}    
\par
}}
}
\noindent Notice that the condition above 
would {\em not} qualify as a Kracht formula when the intended target propositional language is the one introduced in \cite{de2024correspondence}; this is because condition (vii) of Definition \ref{def:inverse_shape} requires all occurrences of  $\overline{a}$-type variables in $\eta$ and $\zeta$ be {\em displayable} when translated as inequalities (i.e.~they occur in isolation  on one side of the inequality). This requirement is satisfied when translating $d \wedge  a \prec b \vee c $ as $d\leq a\sip(b\vee c)$ as we did above; however, when targeting the language of \cite{de2024correspondence}, the atomic formula $d \wedge  a \prec b \vee c $ can be translated either as $\Diamond (d\wedge a)\leq b\vee c$ or as $d\wedge a\leq \blacksquare (b\vee d)$, and in each case, since  conjunction is not in general residuated in the language of (modal) distributive lattices, that occurrence of $d$ is not displayable, which implies that the general algorithm for computing the equivalent axiom will halt and report failure.
A similar argument shows that the condition   which was shown to be equivalent to the `G\"odel-Dummet axiom' in Section \ref{sec:examples on sip} would also violate item (vii) of Definition \ref{def:inverse_shape} when the intended target propositional language is the one introduced in \cite{de2024correspondence}. 
\noindent The examples above show that 
this language contributes to widen the scope of logical/algebraic characterizations of conditions on normative systems in a principled way.

\section{Modelling deontic reasoning}\label{sec:Models}
 Earlier on, we discussed how $a \sip b$ can be understood as the weakest side condition, or {\em context}, under which $a$ normatively implies $b$, and hence, adding  $\sip$ makes the language capable to describe situations in which obligations and permissions may change based on context, capacity, or other factors.

Also, in many real-life situations, different conditional obligations have different levels of priority based e.g.~on  urgency, ethics, or legal requirements. For example, `if a hospital is overcrowded, doctors should treat patients based on urgency, although, by hospital policy, patients should be visited in order of arrival.' In law, `if a lawyer learns confidential information that could prevent a serious crime, they should report it, even though client confidentiality is generally a top priority.'  We can use  slanted Heyting algebras  to  formalize conditional obligations with  different levels of priority; for instance,  consider the propositions `you are a doctor' ($doctor$), `you visit patients according to their order of arrival' ($order$), and `you save lives' ($save$).  The  fact that saving lives has higher priority  for a doctor than following the order of arrival of patients can be formalized by requiring that the obligation $doctor \prec save$   hold under any context in which $doctor \prec order$   holds. That is, for any context $c$, if  $c \wedge doctor \prec order$, then  $c \wedge doctor \prec save$, which is equivalent to the inequality  $doctor \sip order \leq doctor \sip save$\footnote{If we want to explicitly model that there is a context $c$ for  which  $c \wedge doctor \prec save$ holds, but  $c \wedge doctor \prec order $ does not hold, then $\leq$ can be replaced by  $<$.}.

 Finally,  the operator $\sip$ allows us to formalize an infinite set of conditional obligations symbolically. For instance,
in the context of  the functioning of an autonomous vehicle, for any  $ 0 < x \in \mathbb{R}$, consider the obligation $speed_x \wedge obst \prec brake$, which reads `if your speed is  greater than or equal to  $x$ kmph and there is an obstacle in front of you, then you must  brake'. Then, the conditional obligation  `If there is an obstacle in front and your speed is {\em strictly greater than} $0$ kmph, then you must brake' is captured by the infinite set of obligations  $\{ speed_x\wedge obst \prec brake \mid 0 < x \in \mathbb{R}\}$. This can be represented equivalently by the identity  $obst \sip brake =Positive$, where $obst \sip brake =\bigvee\{speed_x\mid 0<x\in \mathbb{R}\}$, and  $Positive$ stands for `speed of vehicle is positive'. Note that by requiring the identity above, we do not require the infinite join $\bigvee \{speed_x \mid 0 <  x\} $ to be a proposition in our normative system (i.e.~an element of the distributive lattice $A$);  however, this join can be represented in terms of propositions $obst$ and $brake$ using 
 the operator $\sip$.

\section{Conclusions and future directions}\label{sec:conclusions}

This paper introduces a framework for modeling  normative  relationships flexibly, where the satisfaction of an obligation  depends on the specific context. The addition of slanted implications  enables to model flexible dependencies that reflect real-world situations in which obligations and permissions may change based on context, capacity, and other factors. 

This work suggests further directions for future research. Firstly, the current approach can be extended to permission systems \cite{Makinson03}: it would be interesting to explore how to model contextual dependencies not only on norms but also on permissions, and on the interaction between norms and permissions. Secondly,  changing the propositional base, and hence  studying these `normative implications' on (co-)Heyting algebras, Boolean algebras, and modal (e.g.~epistemic, temporal) algebras could be valuable. Finally, in Footnote \ref{footnote}, we briefly mentioned the possibility of introducing generalized implications associated with relations $\prec$ with weaker properties than subordination relations. Exploring these settings is another interesting direction. 

\Appendix 
\section{More examples of correspondence and inverse correspondence}
In section, we collect some more examples of correspondence and inverse correspondence on slanted Heyting and co-Heyting algebras.

The following axiom is the  normative counterpart of the weak excluded middle:

  {{\centering
           \begin{tabular}{cll}
                & $\top\leq \neg a \vee \neg \neg a$ \\
        iff  & $ k\leq \neg a\implies \top \leq \neg a \vee \neg k$&denseness\\
        iff  & $ c\leq \neg a\implies \top \leq \neg a \vee \neg c$& compactness\\
        iff  & $ c\leq \neg a \implies \exists d\exists e (\top \leq d\vee e\ \&\ d\leq \neg a\ \&\ e\leq \neg c)$& compactness\\
        iff  & $ c\wedge a\prec \bot \implies \exists d\exists e (\top \leq d\vee e\ \&\ d\wedge a\prec \bot \ \&\ e\wedge c\prec\bot)$& lemma \ref{lem:negation properties}(i)\\
           \end{tabular}
 \par
 }}

 \noindent This condition can be interpreted as, for any propositions $a$ and $c$, if they give normative inconsistency together, then there exist propositions $d$ and $e$ s.t. $d$ or $e$ is always the case, and both  $d$ and $a$ and $e$ and $c$ together give normative inconsistencies.  

The following axiom is the  normative counterpart of the Kreisel-Putnam axiom:

{\small
{{\centering
           \begin{tabular}{cll}
                &  $\scn a \sip (b \vee c)\leq  (\scn a\sip b)\vee ( \scn a\sip c)$\\
            iff & $d\leq \scn a \sip (b \vee c)\implies d\leq  (\scn a\sip b)\vee ( \scn a\sip c)$&Prop. \ref{prop:background can-ext}(ii)\\
            iff & $d\leq \scn a \sip (b \vee c)\implies \exists e\exists f (d\leq e\vee f\ \&\ e\leq (\scn a\sip b)\ \&\ f\leq ( \scn a\sip c))$&compactness\\
            iff & $d\leq g \sip (b \vee c)\ \&\  \sim a\leq g \implies \exists e\exists f\exists h \exists i (d\leq e\vee f\ \&\ e\leq h\sip b$&\\
            & \&\ $\sim a\leq h\ \&\ f\leq  i \sip c\ \&\  \sim a\leq i)$&lemma \ref{lemma:sip and prec}(ii)\\
            iff & $d\leq g \sip (b \vee c)\ \&\  \top\prec g\vee a \implies \exists e\exists f\exists h \exists i (d\leq e\vee f$&\\
            &\&\ $e\leq h\sip b\ \&\  \top\prec a\vee h\ \&\ f\leq  i \sip c\ \&\  \top\prec a\vee i)$&lemma \ref{lem:negation properties}(i)\\
            iff & $d\wedge g \prec b \vee c\ \&\  \top\prec g\vee a \implies \exists e\exists f\exists h \exists i (d\leq e\vee f$&\\
            &\&\ $e\wedge h\prec b\ \&\  \top\prec a\vee h\ \&\ f\wedge  i \prec c\ \&\  \top\prec a\vee i)$&Prop. \ref{prop:back and forth ast}(i)\\
           \end{tabular}
\par
}}
}
\noindent This condition can be interpreted as, 
 if for all $a$, $b$, $c$ and $d$, some  $g$ exists s.t. $g$  or $a$ is an unconditional obligation and $d$ and $g$ together normatively imply $b$ or $c$, then there exist $e$, $f$,$h$, and $i$ s.t. $d$ implies $e $ or $f$, $a$ or $h$ and $a$ or $i$ are both  unconditional obligations, and $e$ and $h$ together  normatively imply $b$ and $f$ and $i$ together  normatively imply $c$. 
 
\bibliographystyle{deon16}
\bibliography{ref}

\end{document}